\documentclass{amsart}
\usepackage{graphicx}
\usepackage{amsmath,amssymb,amsfonts}
\newtheorem{theorem}{Theorem}[section]
\newtheorem{lemma}[theorem]{Lemma}
\newtheorem{coro}[theorem]{Corollary}

\theoremstyle{definition}
\newtheorem{definition}[theorem]{Definition}

\numberwithin{figure}{section}
\theoremstyle{remark}

\numberwithin{equation}{section}

\def \A{\mathcal{A}}

\def \la{\lambda}
\def \I{\mathcal{I}}

\begin{document}
\title[Least H-eigenvalue of Hypergraphs]{The least H-eigenvalue of adjacency tensor of hypergraphs with cut vertices}

\author[Y.-Z. Fan]{Yi-Zheng Fan*}
\address{School of Mathematical Sciences, Anhui University, Hefei 230601, P. R. China}
%\curraddr{Department of Mathematics and Statistics,Case Western Reserve University, Cleveland, Ohio 43403}
\email{fanyz@ahu.edu.cn}
\thanks{*The corresponding author. This work was supported by National Natural Science Foundation of China (Grant No. 11871073, 11771016).
}

\author[Z. Zhu]{Zhu Zhu}
\address{School of Mathematical Sciences, Anhui University, Hefei 230601, P. R. China}
\email{2937242741@qq.com}

\author[Y. Wang]{Yi Wang}
\address{School of Mathematical Sciences, Anhui University, Hefei 230601, P. R. China}
\email{wangy@ahu.edu.cn}


\begin{abstract}
Let $G$ be a connected hypergraph with even uniformity, which contains cut vertices.
Then $G$ is the coalescence of two nontrivial connected sub-hypergraphs (called branches) at a cut vertex.
Let $\A(G)$ be the adjacency tensor of $G$.
The least H-eigenvalue of $\A(G)$ refers to the least real eigenvalue of $\A(G)$ associated with a real eigenvector.
In this paper we obtain a perturbation result on the least H-eigenvalue of $\A(G)$
   when a branch of $G$ attached at one vertex is relocated to another vertex, and
characterize the unique hypergraph whose least H-eigenvalue attains the minimum among all hypergraphs in a certain class of hypergraphs which contain a fixed connected hypergraph.
\end{abstract}



\maketitle

\section{Introduction}
%In 2005 Lim \cite{Lim} and Qi \cite{Qi2} independently introduced the eigenvalues of tensors or hypermatrices.
Recently, the spectral hypergraph theory developed rapidly, the adjacency tensors $\A(G)$ \cite{CD} of uniform hypergraphs $G$
were introduced to investigating the structure of hypergraphs, just like adjacency matrices to simple graphs.
As $\A(G)$ is nonnegative, by using Perron-Frobenius theorem \cite{CPZ, FGH, YY1, YY2, YY3}, many results about its spectral radius are presented \cite{CQZ,CD,FBH,FTPL,LSQ,LM}.
%The Laplacian tensor $\L(G)$ is a Z-tensor and singular M-tenor \cite{ZQZ}, and has zero as its least H-eigenvalue, where
%an eigenvalue is called an \emph{H-eigenvalue} if it is associated with a real eigenvector, and is called an \emph{N-eigenvalue} otherwise.

For the least H-eigenvalue of $\A(G)$ of a $k$-uniform connected hypergraph $G$,
Khan and Fan \cite{KF2} discussed the limit points of the least H-eigenvalue of $\A(G)$ when $G$ is non-odd-bipartite.
Let $\rho(G)$ be the spectral radius of $\A(G)$.
Shao et al. \cite{SSW} proved that
the $-\rho(G)$ is an H-eigenvalue of $\A(G)$ if and only if $k$ is even and $G$ is odd-bipartite.
Some other equivalent conditions are summarized in \cite{FY}.
Note that $-\rho(G)$ is an eigenvalue  of $\A(G)$ if and only if $k$ is even and $G$ is odd-colorable \cite{FY}.
So, there exist odd-colorable but non-odd-bipartite hypergraphs \cite{FKT,Ni}, for which $-\rho(G)$ is an N-eigenvalue.
An odd-colorable hypergraph has a symmetric spectrum.
One can refer \cite{FHB} for more results on the spectral symmetry of nonnegative tensors and hypergraphs.
%Hu and Qi \cite{Hu} discuss the H-eigenvectors of zero eigenvalue of $\L(G)$ or $\Q(G)$ related to the even or odd-bipartitions of $G$;
%they also use N-eigenvectors of zero eigenvalue of $\L(G)$ or $\Q(G)$ to discuss some kinds of partition of $G$,
%where an eigenvector is called an \emph{H}-(or \emph{N}-)\emph{eigenvector} if it can (or cannot) be scaled into a real vector.

To our knowledge, the least H-eigenvalue of $\A(G)$ receives little attention except the above work.
%In this paper, we focus on the least H-eigenvalue of $\A(G)$.
%Qi \cite{Qi2} proved that each eigenvalue of $\Q(G)$ of a $k$-uniform hypergraph $G$ has a nonnegative real part by using Gershgorin disks, which implies that
%the least H-eigenvalue of  $\Q(G)$ is at least zero, and is zero if and only if $k$ is even and $G$ is odd-bipartite.
If $k$ is even, then the least H-eigenvalue of $\A(G)$ is a solution of minimum problem over a real unit sphere; see Eq. (\ref{form2}).
So, throughout of this paper, when discussing the least H-eigenvalue of $\A(G)$,
we always assume that \emph{$G$ is connected with even uniformity $k$}.
For convenience, the least H-eigenvalue of $\A(G)$ is simply called the \emph{least eigenvalue} of $G$
 and the corresponding H-eigenvectors are called the \emph{first eigenvectors} of $G$.

In this paper we give a perturbation result on the least H-eigenvalue of $\A(G)$
   when a branch of $G$ attached at one vertex is relocated to another vertex, and
characterize the unique hypergraph whose least H-eigenvalue attains the minimum among all hypergraphs in a certain class of hypergraphs which contain a fixed connected hypergraph.
%Finally we present some upper bounds of the least eigenvalue and prove that zero is the least limit point of the least $\Q$-eigenvalues of connected non-odd-bipartite hypergraphs.
The perturbation result in this paper is a generalization of that on the least eigenvalue of the adjacency matrix of a simple graph in \cite{FWG}.

\section{Preliminaries}

\subsection{Tensors and eigenvalues}
A real {\it tensor} (also called \emph{hypermatrix}) $\A=(a_{i_{1} i_2 \ldots i_{k}})$ of order $k$ and dimension $n$ refers to a
  multi-dimensional array with entries $a_{i_{1}i_2\ldots i_{k}}\in \mathbb{R}$
  for all $i_{j}\in [n]:=\{1,2,\ldots,n\}$ and $j\in [k]$.
Clearly, if $k=2$, then $\A$ is a square matrix of dimension $n$.
The tensor $\A$ is called \emph{symmetric} if its entries are invariant under any permutation of their indices.

 Given a vector $x\in \mathbb{C}^{n}$, $\A x^{k} \in \mathbb{C}$ and $\A x^{k-1} \in \mathbb{C}^n$, which are defined as follows:
\begin{align*}
\A x^{k} & =\sum_{i_1,i_{2},\ldots,i_{k}\in [n]}a_{i_1i_{2}\ldots i_{k}}x_{i_1}x_{i_{2}}\cdots x_{i_k},\\
(\A x^{k-1})_i & =\sum_{i_{2},\ldots,i_{k}\in [n]}a_{ii_{2}\ldots i_{k}}x_{i_{2}}\cdots x_{i_k}, i \in [n].
\end{align*}

Let $\mathcal{I}=(i_{i_1i_2\ldots i_k})$ be the {\it identity tensor} of order $k$ and dimension $n$, that is, $i_{i_{1}i_2 \ldots i_{k}}=1$ if
   $i_{1}=i_2=\cdots=i_{k} \in [n]$ and $i_{i_{1}i_2 \ldots i_{k}}=0$ otherwise.
%In 2005, Lim \cite{Lim} and Qi \cite{Qi2} independently introduced the eigenvalues of tensors.

\begin{definition}[\cite{Lim,Qi2}]
Let $\A$ be a real tensor of order $k$ dimension $n$.
For some $\lambda \in \mathbb{C}$, if the polynomial system $(\lambda \mathcal{I}-\A)x=0$,
or equivalently $\A x^{k-1}=\lambda x^{[k-1]}$, has a solution $x \in \mathbb{C}^{n}\backslash \{0\}$,
then $\lambda $ is called an \emph{eigenvalue} of $\A$ and $x$ is an \emph{eigenvector} of $\A$ associated with $\lambda$,
where $x^{[k-1]}:=(x_1^{k-1}, x_2^{k-1},\ldots,x_n^{k-1})$.
\end{definition}

 In the above definition, $(\lambda,x)$ is called an \emph{eigenpair} of $\A$.
 If $x$ is a real eigenvector of $\A$, surely the corresponding eigenvalue $\lambda$ is real.
In this case, $\lambda$ is called an {\it H-eigenvalue} of $\A$.
A real symmetric tensor $\A$ of even order $k$ is called \emph{positive semidefinite}  (or \emph{positive definite}) if for any $x\in \mathbb{R}^{n} \backslash \{0\}$,
$\A x^k \ge 0$ (or $\A x^k > 0$).
Denote by $\lambda_{\min}(\A)$ the least H-eigenvalue of $\A$.

%\begin{lemma}[\cite{Qi2}, Theorem 5] \label{semi}
%Let $\A$ be a real symmetric tensor of order $k$ and dimension $n$, where $k$ is even.
%Then $\A$ always has H-eigenvalues, and
%$$\lambda_{\min}(\A)=\min\{\A x^k: x\in \mathbb{R}^{n}, \|x\|_k=1\},$$
%where $\|x\|_k=\left(\sum_{i=1}^{n} |x_i^k|\right)^{1 \over k}$.
%Furthermore, $x$ is an optimal solution of the above optimization
%if and only if it is an eigenvector of $\A$ associated with $\lambda_{\min}(\A)$.
%\end{lemma}

\begin{lemma}[\cite{Qi2}, Theorem 5] \label{semi}
Let $\A$ be a real symmetric tensor of order $k$ and dimension $n$, where $k$ is even.
Then the following results hold.

\begin{enumerate}

\item $\A$ always has H-eigenvalues, and $\A$ is positive definite (or positive semidefinite) if and only if its least H-eigenvalue is positive (or nonnegative).

\item $\lambda_{\min}(\A)=\min\{\A x^k: x\in \mathbb{R}^{n}, \|x\|_k=1\}$,
where $\|x\|_k=\left(\sum_{i=1}^{n} |x_i|^k\right)^{1 \over k}$.
Furthermore, $x$ is an optimal solution of the above optimization
if and only if it is an eigenvector of $\A$ associated with $\lambda_{\min}(\A)$.
\end{enumerate}
\end{lemma}

\subsection{Uniform hypergraphs and eigenvalues}
A {\it hypergraph} $G=(V(G),E(G))$ is a pair consisting of a vertex set $V(G)=\{v_1,v_2,\ldots,v_n\}$ and an edge set $E(G)=\{e_{1},e_2,\ldots,e_m\}$,
where $e_j\subseteq V(G)$ for each $j\in[m]$.
If $|e_j|=k$ for all $j\in[m]$, then $G$ is called a {\it $k$-uniform} hypergraph.
The {\it degree} $d_G(v)$ or simply $d(v)$ of a vertex $v \in V(G)$ is defined as $d(v)=|\{e_{j}:v\in e_{j}\}|$.
%The {\it order} of  $G $ is the cardinality of $V(G)$, denoted by $\nu(G)$, and its {\it size} is the cardinality of $E(G)$, denoted by $\varepsilon(G)$.
A {\it walk} in a $G$ is a sequence of alternate vertices and edges: $v_0e_1v_1e_2 \ldots e_lv_l$,
where ${v_i, v_{i+1}}\in e_i$ for $i = 0, 1, \ldots, l-1$.
A walk is called a {\it path} if all the vertices and edges appeared on the walk are distinct.
A hypergraph $G$ is called {\it connected} if every two vertices of $G$ are connected by a walk or path.

If a hypergraph is both connected and acyclic, it is called a \emph{hypertree}.
The {\it $k$-th power} of a simple graph $H$, denoted by $H^k$,
is obtained from $H$ by replacing each edge (a $2$-set) with a $k$-set by adding $(k-2)$ additional vertices \cite{HQS}.
The $k$-th power of a tree is called \emph{power hypertree}, which is surely a $k$-uniform hypertree.
In particular, the $k$-th power of a star (as a simple graph) with $m$ edges is called a \emph{hyperstar},
denote by $S_m^k$.
In a $k$-th power hypertree $T$, an edge is called a \emph{pendent edge} of $T$ if it contains $k-1$ vertices of degree one,
which are called the \emph{pendent vertices} of $T$.

\begin{lemma}[\cite{Ber1}]\label{tree}
If $G$ is a connected $k$-uniform hypergraph with $n$ vertices and $m$ edges,
then $G$ is a hypertree if and only if $m=\frac{n-1}{k-1}.$
\end{lemma}

\begin{definition}[\cite{Hu}]
Let $k$ be even.
A $k$-uniform hypergraph $G=(V,E)$ is called \emph{odd-bipartite},
if there exists a bipartition $\{V_1,V_2\}$ of $V$ such that each edge of $G$ intersects $V_1$ (or $V_2$) in an odd number of vertices (such bipartition is called
an \emph{odd-bipartition} of $G$);
otherwise, $G$ is called \emph{non-odd-bipartite}.
\end{definition}

The odd-bipartite hypergraphs are considered as generalizations of the ordinary bipartite graphs.
%The odd-bipartition is closely related to odd-traversal \cite{Ni}.
The examples of non-odd-bipartite hypergraphs can been found in \cite{KF, Ni}.
Note that the odd-bipartite hypergraphs are also called odd-transversal hypergraphs in \cite{Ni}.
As a special case of generalized power hypergraphs defined in \cite{KF},
  the $k$-uniform hypergraph $G^{k,{k \over 2}}$ is obtained from a simple graph $G$ by blowing up each vertex into an ${k \over 2}$-set and preserving the adjacency relation, where $k$ is even.
It was shown that  $G^{k,{k \over 2}}$ is odd-bipartite if and only if $G$ is bipartite \cite{KF}.

Let $G$ be a $k$-uniform hypergraph on $n$ vertices $v_1,v_2,\ldots,v_n$.
The {\it adjacency tensor} of $G$ \cite{CD} is defined as $\mathcal{A}(G)=(a_{i_{1}i_{2}\ldots i_{k}})$, an order $k$ dimensional $n$ tensor, where
\[a_{i_{1}i_{2}\ldots i_{k}}=\left\{
 \begin{array}{ll}
\frac{1}{(k-1)!}, &  \mbox{if~} \{v_{i_{1}},v_{i_{2}},\ldots,v_{i_{k}}\} \in E(G);\\
  0, & \mbox{otherwise}.
  \end{array}\right.
\]
%
% Let $\mathcal{D}(G)$ be a diagonal tensor of order $k$ and dimension $n$, where $d_{i\ldots i}=d(v_i)$ for $i \in [n]$.
%The tensor $\Q(G)=\mathcal{D}(G)+\A(G)$ is called the {\it signless Laplacian tensor} of $G$ \cite{Qi}.
Observe that the adjacency tensor of a hypergraph is symmetric.

Let $x=(x_1, x_2,\ldots,x_n) \in \mathbb{C}^n$.
Then $x$ can be considered as a function defined on the vertices of $G$,
 that is, each vertex $v_i$ is mapped to $x_i=:x_{v_ i}$.
 If $x$ is an eigenvector of $\A(G)$, then it defines on $G$ naturally, i.e., $x_v$ is the entry of $x$ corresponding to $v$.
 If $G_0$ is a sub-hypergraph of $G$, denote by $x|_{G_0}$ the restriction of $x$ on the vertices of $G_0$, or a subvector of $x$ indexed by the vertices of $G_0$.

 Denote by $E_G(v)$, or simply $E(v)$, the set of edges of $G$ containing $v$.
 For a subset $U$ of $V(G)$, denote $x^U:=\Pi_{v \in U} x_u$.
 Then we have
 \begin{equation}\label{form}
  \A(G)x^k = \sum_{e\in E(G)}kx^e,
\end{equation}
 and for each $v \in V(G)$,
 $$(\A(G)x^{k-1})_v=\sum_{e\in E(v)}x^{e\backslash \{v\}}.$$
So the eigenvector equation $\A(G)x^{k-1}=\la x^{[k-1]}$ is equivalent to that for each $v \in V(G)$,
\begin{equation}\label{eigen}
\la x_v^{k-1}=\sum_{e\in E(v)}x^{e\backslash \{v\}}.
\end{equation}
From Lemma \ref{semi}(2), if $k$ is even, then $\la_{\min}(G):=\lambda_{\min}(\A(G))$ can be expressed as
\begin{equation}\label{form2}
\lambda_{\min}(G)=\min_{x\in \mathbb{R}^{n}, \|x\|_k=1}\sum_{e\in E(G)}kx^e.
\end{equation}

Note than if $k$ is odd, the Eq. (\ref{form2}) does not hold.
The reason is as follows. If $G$ contains at least one edge, then by Perron-Frobenius theorem,
the spectral radius $\rho(\A(G))$ of $\A(G)$ is positive associated with a unit nonnegative eigenvector $x$.
Now $$-\rho(\A(G))\le \lambda_{\min}(G)\le \A(G)(-x)^k=-\A(G)x^k=-\rho(\A(G)),$$
so $\lambda_{\min}(G)=-\rho(\A(G))$, which implies that $k$ is even and $G$ is odd-bipartite \cite{SSW, FY}, a contradiction.

\begin{lemma}\label{xuv}
Let $G$ be a $k$-uniform hypergraph, and $(\lambda,x)$ be an eigenpair of $\A(G)$.
If $E(u)=E(v)$ and $\la \ne 0$, then $x_u^k=x_v^k$.
\end{lemma}

\begin{proof}
Consider the eigenvector equation of $x$ at $u$ and $v$ respectively,
$$\lambda x_u^k=\sum_{e\in E(u)}x^e, \;\lambda x_v^k=\sum_{e\in E(v)}x^e.$$
As $E[u]=E[v]$, $\sum_{e\in E(u)}x^e=\sum_{e\in E(v)}x^e$.
The result follows.
\end{proof}

%Qi \cite{Qi} show that the least H$^+$-eigenvalue of $\Q(G)$ of a $k$-uniform hypergraph $G$ is exactly the minimum degree of $G$, denoted by $\delta(G)$.
%If $k$ is even, then $\lambda_{\min}(G) \le \delta(G)$.
%We give an improvement as follows.

\begin{lemma}\label{mindegree}
Let $G$ be a $k$-uniform hypergraph with at least one edge, where $k$ is even.
Let $G_0$ be a induced sub-hypergraph of $G$.
 Then $\lambda_{\min}(G) \le \la_{\min}(G_0)$.
 In particular, $\lambda_{\min}(G) \le -1$.
\end{lemma}

\begin{proof}
Let $x$ be a unit first eigenvector of $G_0$.
Define $\bar{x}$ on $G$ such that $\bar{x}(v)=x(v)$ if $v \in V(G_0)$ and $\bar{x}(v)=0$ otherwise.
Then by Lemma \ref{semi},
$$\la_{\min}(G_0)=\A(G_0) x^k = \A(G) \bar{x}^k \ge \min_{x\in \mathbb{R}^{n}, \|x\|_k=1} \A(G) x^k=\lambda_{\min}(G).$$
If letting $G_0$ be an edge, then $\la_{\min}(G_0)=-1$. The result follows.
\end{proof}

\section{Properties of the first eigenvectors}
We will give some properties of the first eigenvectors of a connected $k$-uniform $G$.
We stress that $k$ is \emph{even} in this and the following section.

Let $G_1$, $G_2$ be two vertex-disjoint hypergraphs, and let $v_1\in V(G_1), v_2\in V(G_2)$.
The \emph{coalescence} of $G_1$, $G_2$ with respect to $v_1, v_2$,
   denoted by $G_1(v_1)\diamond G_2(v_2)$, is obtained from $G_1$, $G_2$ by identifying $v_1$ with $v_2$ and forming a new vertex $u$.
The graph $G_1(v_1)\diamond G_2(v_2)$ is also written as $G_1(u)\diamond G_2(u)$.
If a connected hypergraph $G$ can be expressed
in the form $G=G_1(u)\diamond G_2(u)$, where $G_1$, $G_2$ are both nontrivial and connected, then 
$u$ is called a \emph{cut vertex} of $G$, and $G_1$ is called a \emph{branch} of $G$ with \emph{root} $u$. 
Clearly $G_2$ is also a branch of $G$ with root $u$ in the above definition.
%Let $x$ be a vector defined on $G$. A branch $H$ of $G$ is called a \emph{zero branch} with respect to $x$ if $x_v=0$ for all $v\in V(H)$; otherwise it
%is called a \emph{nonzero branch} with respect to $x$.

\begin{lemma}\label{prop1}
Let $G=G_0(u)\diamond H(u)$ be a connected $k$-uniform hypergraph.
Let $x$ be a first eigenvector of $G$. Then the following results hold.

\begin{enumerate}
  \item  If $H$ is odd-bipartite, then $x^e \le 0$ for each $e\in E(H)$, and there exists a first eigenvector of $G$
  　　　　　　　　　　such that it is nonnegative on one part and nonpositive on the other part for any odd-bipartition of $H$.
  \item  If $x_u=0$, then $\sum_{e\in E_{G_0}(u)}x^{e\backslash\{u\}}=\sum_{e\in E_{H}(u)}x^{e\backslash\{u\}}=0$.
  If further $H$ is odd-bipartite, then $x^{e\backslash\{u\}}=0$ for each $e \in E_H(u)$.
  \end{enumerate}
\end{lemma}

\begin{proof}
Let $\{U,W\}$ be an odd-bipartition of $H$, where $u\in U$.
Without loss of generality, we assume that $\|x\|_k=1$ and $x_u\ge 0$.
Let $\tilde{x}$ be such that
\[
\tilde{x}_v=
\begin{cases}
x_v,&\text{if } v\in V(G_0)\backslash\{u\};\\
|x_v|,&\text{if } v\in U;\\
-|x_v|,&\text{if } v\in W.
\end{cases}
\]
Note that $\|\tilde{x}\|_k=\|x\|_k=1$, and for each $e \in E(H)$, $\tilde{x}^e \le x^e$ with equality if and only if $x^e \le 0$.
%\begin{itemize}
%    \item [{\rm(a)}] $\tilde{x}_e^{k}=x_e^{k}$.
%  \item [{\rm(b)}] $\tilde{x}^e \le x^e$ with equality if and only if $x^e \le 0$.
%\end{itemize}

We prove the assertion (1) by a contradiction.
Suppose that there exists an edge $e \in E(H)$ such that $x^{e}>0$.
Then $\tilde{x}^{e}<x^{e}$.
By Eq. (\ref{form2}), we have
\[
\lambda_{\min}(G)\le  \A(G)\tilde{x}^k < \A(G)x^k=\lambda_{\min}(G),
\]
a contradiction.
So $x^e \le  0$ for each $e\in E(H)$, and $\tilde{x}$ is also a first eigenvector as $\A(G)\tilde{x}^k = \A(G)x^k$.
The assertion (1) follows.

For the assertion (2), let $\bar{x}$ be such that
\[
\bar{x}_v=
\begin{cases}
x_v,&\text{if } v\in V(G_0)\backslash\{u\};\\
-x_v,&\text{if } v\in V(H).\\
\end{cases}
\]
Then $\bar{x}$ is also a first eigenvector of $G$ as $\A(G)\bar{x}^k = \A(G)x^k$.
Note that $x_u=0$ and consider the eigenvector equation Eq. (\ref{eigen}) of $x$ and $\bar{x}$ at $u$, respectively.
$$\lambda_{\min}(G) x_u^{k-1}=0=\sum_{e\in E_{G_0}(u)}x^{e\backslash \{u\}}+\sum_{e\in E_H(u)}x^{e\backslash \{u\}},$$
\begin{eqnarray*}
  \lambda_{\min}(G) \bar{x}_u^{k-1}=0 &=& \sum_{e\in E_{G_0}(u)}\bar{x}^{e\backslash \{u\}}+\sum_{e\in E_H(u)}\bar{x}^{e\backslash \{u\}} \\&=& \sum_{e\in E_{G_0}(u)}x^{e\backslash \{u\}}-\sum_{e\in E_H(u)}x^{e\backslash \{u\}}.
\end{eqnarray*}
Thus $\sum_{e\in E_{G_0}(u)}x^{e\backslash\{u\}}=\sum_{e\in E_H(u)}x^{e\backslash \{u\}}=0$.

If further $H$ is odd-bipartite, applying the above result to $\tilde{x}$ (also a first eigenvector of $G$), we have
$\sum_{e\in E_H(u)}\tilde{x}^{e\backslash \{u\}}=0$.
As $\tilde{x}^{e\backslash \{u\}} \le 0$ for each edge $e\in E_H(u)$, $\tilde{x}^{e\backslash\{u\}}=0$ for each $e \in E_H(u)$.
The assertion (2) follows by the definition of $\tilde{x}$.
\end{proof}

\begin{lemma}\label{prop2}
Let $G=G_0(u)\diamond H(u)$ be a connected $k$-uniform hypergraph.
Then $$\lambda_{\min}(G_0)\ge \lambda_{\min}(G),$$
with equality if and only if for any first eigenvector $y$ of $G_0$, $y_u=0$ and $\tilde{y}$ is a first eigenvector of $G$,
where $\tilde{y}$ is defined by
\[
\tilde{y}_v=
\begin{cases}
y_v,&\text{if } v\in V(G_0);\\
0,&\text{otherwise.}
\end{cases}
\]
\end{lemma}

\begin{proof}
Suppose that $y$ is a first eigenvector of $G_0$, $\|y\|_k=1$.
Let $e_0$ be an edge of $H$ containing $u$.
Define $\bar{y}$ by
\[
\bar{y}_v=
\begin{cases}
y_v,&\text{if } v\in V(G_0);\\
-y_{u},&\text{if } v\in e_0 \backslash \{u\};\\
0,&\text{otherwise.}
\end{cases}
\]
Then $\|\bar{y}\|^k_k=1+(k-1)y_u^k$, and
\begin{eqnarray*}
  \A(G)\bar{y}^k &=& \sum_{e\in E(G_0)}k\bar{y}^e+\sum_{e\in E(H)}k\bar{y}^e\\
            &=& \A(G_0)y^k -ky_u^k\\
            &=& \lambda_{\min}(G_0)-ky_u^k.
\end{eqnarray*}
By Eq. (\ref{form2}), we have
$$\lambda_{\min}(G)\le  \frac{\A(G)\bar{y}^k} {\|\bar{y}\|_k^k}=\frac{\lambda_{\min}(G_0)-ky_u^k} {1+(k-1)y_u^k}\le \lambda_{\min}(G_0),$$
where the first equality holds if and only if $\bar{y}$ is also a first eigenvector of $G$, and the second equality holds if and only if $y_u=0$.
The result now follows.
\end{proof}

\begin{coro}\label{prop3}
Let $G=G_0(u)\diamond H(u)$ be a connected $k$-uniform hypergraph.
\begin{enumerate}
  \item If $y$ is a first eigenvector of $G_0$ with $y_u\neq 0$, then
$$\lambda_{\min}(G_0)>\lambda_{\min}(G).$$
  \item If $x$ is a first eigenvector of $G$ such that $x_u=0$ and $x|_{G_0} \ne 0$, then
$$\lambda_{\min}(G_0)=\lambda_{\min}(G).$$
\end{enumerate}
\end{coro}

\begin{proof}
By Lemma \ref{prop2}, we can get the assertion (1) immediately.
Let $x$ be a first eigenvector of $G$ as in (2).
By Lemma \ref{prop1}(2), $\sum_{e\in E_{G_0}(u)}x^{e\backslash\{u\}}=0$.
Considering the eigenvector equation (\ref{eigen}) of $x$ at each vertex of $V(G_0)$, we have
$$\A(G_0)(x|_{G_0})^{k-1}=\lambda_{\min}(G)x^{k-1}.$$
So $x|_{G_0}$, the restriction of $x$ on $G_0$, is an eigenvector of $\A(G_0)$ associated with the eigenvalue $\lambda_{\min}(G)$.
The result follows by Lemma \ref{prop2}.
\end{proof}

\begin{lemma}\label{bh}
Let $G=G_0(u)\diamond H(u)$ be a connected $k$-uniform hypergraph.
If $x$ is a first eigenvector of $G$, then
\begin{equation}\label{beta}
\alpha_H(x):=\sum_{e\in E_H(u)} x^e \le 0.
\end{equation}
Furthermore, if $\alpha_H(x)=0$ and $x|_{G_0} \ne 0$, then $x_u=0$ and $\lambda_{\min}(G_0)=\lambda_{\min}(G)$;
or equivalently if $x_u \ne 0$, then $\alpha_H(x)<0$.
\end{lemma}

\begin{proof}
Let $\la:=\la_{\min}(G)$.
By Eq. (\ref{eigen}),
for each $v\in V(G_0)\backslash \{u\}$,
\begin{equation}\label{sem}
((\A(G)-\lambda \I) x^{k-1})_v= ((\A(G_0)-\lambda\I)(x|_{G_0})^{k-1})_v=0.
\end{equation}
For the vertex $u$,
\begin{eqnarray*}
\lambda x_u^{k-1} =(\A(G)x^{k-1})_u &= & \sum_{e\in E_{G_0}(u)}x^{e\backslash \{u\}}+ \sum_{e\in E_H(u)} x^{e\backslash \{u\}}\\
&= & (\A(G_0)(x|_{G_0})^{k-1})_u + \sum_{e\in E_H(u)} x^{e\backslash \{u\}}.
%((\Q(G_0)-\lambda\I)x|_{G_0})_u &=& d_{G_0}(u)x_u^{k-1}+ \sum_{e\in E_G(u)}x^{e\backslash \{u\}}-\lambda x_u^{k-1}\\
\end{eqnarray*}
So,
\begin{equation}\label{u}
((\A(G_0)-\lambda\I)(x|_{G_0})^{k-1})_u=-\sum_{e\in E_H(u)} x^{e\backslash \{u\}}.
\end{equation}

By Lemma \ref{prop2} and Lemma \ref{semi}(1), $\A(G_0)-\lambda\I$ is positive semidefinite.
Then $(\A(G_0)-\lambda\I)y^k \ge 0$ for any real and nonzero $y$.
So, by Eq. (\ref{sem}) and Eq. (\ref{u}), we have
\begin{eqnarray*}
0 \le (\A(G_0)-\lambda\I)(x|_{G_0})^k &=& (x|_{G_0})^\top ((\A(G_0)-\lambda\I)(x|_{G_0})^{k-1} \\
                                        &=& -x_u  \sum_{e\in E_H(u)} x^{e\backslash \{u\}} \\
                                         &=& -\alpha_H(x).
\end{eqnarray*}
So we have $\alpha_H(x)\le 0$.

Suppose that $\alpha_H(x)=0$ and $x|_{G_0} \ne 0$.
If $x_u=0$, by Corollary \ref{prop3}(2), $\lambda_{\min}(G_0)=\lambda_{\min}(G)$.
If $x_u\neq 0$, then $\sum_{e\in E_H(u)} x^{e\backslash \{u\}}=0$.
By Eq. (\ref{sem}) and Eq. (\ref{u}), $(\lambda_{\min}(G),x|_{G_0})$ is an eigenpair of $\A(G_0)$, implying that $\lambda_{\min}(G)=\lambda_{\min}(G_0)$ by Lemma \ref{prop2}.
However, $(x|_{G_0})_u =x_u \ne 0$, a contradiction to Corollary \ref{prop3}(1).
\end{proof}

\section{Perturbation of the least eigenvalues}

We will give a perturbation result on the least eigenvalues under relocating a branch.
%We need introduce an operation called \emph{relocating}.
Let $G_0$, $H$ be two vertex-disjoint hypergraphs, where $v_1$, $v_2$ are two distinct vertices of $G_0$, and $u$ is a vertex of $H$ (called the \emph{root} of $H$).
Let $G = G_0(v_2)\diamond H(u)$ and $\tilde{G} =G_0(v_1)\diamond H(u)$.
We say that $\tilde{G}$ is obtained from $G$ by relocating $H$ rooted at $u$ from $v_2$ to $v_1$; see Fig. \ref{relo}.

\begin{figure}[htbp]
\includegraphics[scale=.8]{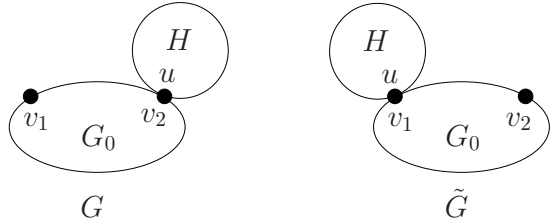}
\caption{\small Relocating $H$ from $v_2$ to $v_1$}\label{relo}
\end{figure}

\begin{lemma}\label{movedge}
Let $G=G_0(v_2)\diamond H(u)$ and $\tilde{G}=G_0(v_1)\diamond H(u)$ be connected $k$-uniform hypergraphs.
If $x$ is a first eigenvector of $G$ such that $|x_{v_1}|\ge |x_{v_2}|$,  then
\begin{equation}\label{le}
\lambda_{\min}(\tilde{G})\le \lambda_{\min}(G),
\end{equation}
with equality if and only if $x_{v_1}= x_{v_2}=0$, and $\tilde{x}$ defined in (\ref{case2}) is a first eigenvector of $\tilde{G}$.
\end{lemma}

\begin{proof}
Let $x$ be a first eigenvector of $G$ such that $\|x\|_k=1$ and $x_{v_1}\ge 0$.
We divide the discussion into three cases.
Denote $\la:=\lambda_{\min}(G)$.

\textbf{Case 1:} $x_{v_2}>0$.
Write $x_{v_1}=\delta x_{v_2}$, where $\delta \ge  1$.
Define $\tilde{x}$ on $\tilde{G}$ by
\begin{equation}\label{evtld1}
\tilde{x}_v=
\begin{cases}
x_v,&\text{if } v\in V(G_0);\\
\delta x_v,&\text{if } v\in V(H)\backslash\{u\}.
\end{cases}
\end{equation}
Then $\|\tilde{x}\|_k^k=1+(\delta^k-1)\sum_{v\in {V(H)\backslash\{u\}}}x_v^k$, and
\begin{eqnarray*}
  \A(\tilde{G})\tilde{x}^k &=& \sum_{e\in E(\tilde{G})}k \tilde{x}^e \\
            &=& \A(G)x^k +(\delta^k-1)\sum_{e\in E(H)}kx^e\\
            &=& \la+(\delta^k-1)\sum_{e\in E(H)}kx^e.
\end{eqnarray*}

By the eigenvector equation of $x$ at each vertex $v\in V(H)\backslash \{u\}$,
\begin{equation}\label{ehv}
\sum_{e\in E_H(v)}x^e=\la x_v^k.
\end{equation}
By the eigenvector equation of $x$ at $u$,
\begin{equation}\label{egu}
\sum_{e\in E_G(u)}x^e=\alpha_H(x)+\gamma_{G_0}(x)=\la x_u^k,
\end{equation}
where $\gamma_{G_0}(x):=\sum_{e\in E_{G_0}(u)}x^e$.
By Eq. (\ref{ehv}) and Eq. (\ref{egu}), we have
$$\gamma_{G_0}(x)+\sum_{e\in E(H)}kx^e=\la \sum_{v\in V(H)}x_v^k.$$
So
\begin{eqnarray*}
  \sum_{e\in E(H)}kx^e &=& \la \sum_{v\in V(H)}x_v^k-\gamma_{G_0}(x) \\
  &=& \la \sum _{v\in V(H)}x_v^k-(\la x_u^k-\alpha_H(x))\\
  &=& \la \sum_{v\in {V(H)\backslash\{u\}}}x_v^k+\alpha_H(x).
\end{eqnarray*}
Thus
\begin{eqnarray*}
   \A(\tilde{G})\tilde{x}^k &=&  \la+(\delta^k-1)\left(\la \sum_{v\in {V(H)\backslash\{u\}}}x_v^k+\alpha_H(x)\right)\\
            &=& \la \left(1+(\delta^k-1)\sum_{v\in {V(H)\backslash\{u\}}}x_v^k\right)+(\delta^k-1)\alpha_H(x)\\
            &=& \la \|\tilde{x}\|_k^k+(\delta^k-1)\alpha_H(x).
\end{eqnarray*}
As $x_{v_2}\ne 0$, $\alpha_H(x) < 0$ by Lemma \ref{bh},
\[
\lambda_{\min}(\tilde{G})\le  \frac{\A(\tilde{G})\tilde{x}^k} {\|\tilde{x}\|_k^k}
=\lambda+\frac{(\delta^k-1)\alpha_H(x)} {\|\tilde{x}\|_k^k}\le \la=\lambda_{\min}(G),
\]
where the first equality holds if and only if $\tilde{x}$ is a first eigenvector of $\tilde{G}$,
and the second equality holds if and only if $\delta=1$, i.e. $x_{v_1}=x_{v_2}$.
%However, by Lemma \ref{bh}, $\alpha_H(x)<0$ as $x_{v_2}\ne 0$.
So, the equality holds if and only if $x_{v_1}=x_{v_2}$ and $\tilde{x}$ is a first eigenvector of $\tilde{G}$.
By the eigenvector equations of $x$ and $\tilde{x}$ at $v_2$ respectively, we will get $\alpha_H(x)=0$, a contradiction.
So, in this case, $\lambda_{\min}(\tilde{G})< \lambda_{\min}(G)$.

\textbf{Case 2:} $x_{v_2}=0$.
First assume $x_{v_1}=0$.
Define $\tilde{x}$ on $\tilde{G}$ by
\begin{equation}\label{case2}
\tilde{x}_v=
\begin{cases}
x_v,&\text{if } v\in V(G_0);\\
x_v,&\text{if } v\in V(H)\backslash\{u\}.
\end{cases}
\end{equation}
Then $\|\tilde{x}\|_k^k=1$,
and
$$\la_{\min}(\tilde{G}) \le \A(\tilde{G})\tilde{x}^k=\A(G)x^k =\la_{\min}(G),$$
with equality if and only if $\tilde{x}$ is a first eigenvector of $\tilde{G}$.
%In this case we surely have $\alpha_H(x)=0$ as $x_u=0$.

Now assume that $x_{v_1}>0$.
By Corollary \ref{prop3}(2) and its proof, $\la_{\min}(G)=\la_{\min}(G_0)$ as $x_{u}=0$ and $x|_{G_0} \ne 0$;
furthermore, $x|_{G_0}$ is a first eigenvector of $G_0$.
By  Corollary \ref{prop3}(1), $\la_{\min}(G_0)>\la_{\min}(\tilde{G})$ as $(x|_{G_0})_{v_1}\ne 0$,
thinking of $v_1$ a coalescence vertex between $G_0$ and $H$ in $\tilde{G}$.
So $\la_{\min}(\tilde{G}) < \la_{\min}(G)$.

\textbf{Case 3:} $x_{v_2}<0$.
Write $x_{v_1}=-\delta x_{v_2}$, where $\delta \ge  1$.
Define $\tilde{x}$ on $\tilde{G}$ by
\begin{equation}\label{evtld2}
\tilde{x}_v=
\begin{cases}
x_v,&\text{if } v\in V(G_0);\\
-\delta x_v,&\text{if } v\in V(H)\backslash\{u\}.
\end{cases}
\end{equation}
By a similar discussion to Case 1 by replacing $\delta$ by $-\delta$,  we also have
$\lambda_{\min}(\tilde{G})< \lambda_{\min}(G)$.
%where the  equality holds if and only if $\tilde{x}$ is a first eigenvector of $\tilde{G}$,
%and $x_{v_1}=-x_{v_2}$ or $\alpha_H(x)=0$.
\end{proof}

Denoted by $\mathcal{T}_m(G_0)$ the class of hypergraphs with each obtained from a fixed connected hypergraph $G_0$ by attaching
some hypertrees at some vertices of $G_0$ respectively
(i.e. identifying a vertex of a hypertree with some vertex of $G_0$ each time) such that the number of its edges equals $\varepsilon(G_0)+m$.
A hypergraph is called a \emph{minimizing hypergraph} in a certain class of hypergraphs if its least eigenvalue attains the minimum
among all hypergraphs in the class.

We will characterize the minimizing hypergraph(s) in $\mathcal{T}_m(G_0)$.
Denote by $G=G_0(u)\diamond S_m^k(u)$ the coalescence of $G_0$ and $S_m^k$ by identifying one vertex of $G_0$ and the central vertex of $S_m^k$ and forming a new vertex $u$.

\begin{theorem}\label{min}
Let $G_0$ be a connected $k$-uniform hypergraph.
If $G$ is a minimizing hypergraph in $\mathcal{T}_m(G_0)$, then $G=G_0(u) \diamond S_m^k(u)$ for a unique vertex $u$ of $G_0$.
\end{theorem}

\begin{proof}
Suppose that $G$ is a minimizing hypergraph in $\mathcal{T}_m(G_0)$, and $G$ has no the structure as desired in the theorem.
We will get a contradiction by the following three cases.

\textbf{Case 1:} $G$ contains hypertrees attached at two or more vertices of $G_0$.
Let $T_1$, $T_2$ be two hypertrees attached at $v_1,v_2$ of $G_0$ respectively.
Let $x$ be a first eigenvector of $G$.
Assume $|x_{v_1}| \ge |x_{v_2}|$.
Relocating $T_2$ rooted at $v_2$ and attaching to $v_1$, we will get a hypergraph $\bar{G} \in \mathcal{T}_m(G_0)$ such
that $\la_{\min}(\bar{G}) \le \la_{\min}(G)$ by Lemma \ref{movedge}.
Repeating the above operation, we finally arrive at a hypergraph $G^{(1)}$ with only one hypertree $T^{(1)}$ attached at one vertex $u_0$ of $G_0$ such that
$\la_{\min}(G^{(1)}) \le \la_{\min}(G)$.

{\bf Case 2:} $T^{(1)}$ contains edges with three or more vertices of degree greater than one, i.e. $T^{(1)}$ is not a power hypertree.
Let $e$ be one of such edges containing $u,v,w$ with $d(u),d(v), d(w)$ all greater than one.
Let $x$ be a first eigenvector of $G^{(1)}$, and assume that $|x_u| \ge |x_w|$.
Relocating the hypertree rooted at $w$ and attaching to $u$, we will get a hypergraph $\hat{G} \in \mathcal{T}_m(G_0)$ such
that $\la_{\min}(\hat{G}) \le \la_{\min}(G^{(1)})$ by Lemma \ref{movedge}.
Repeating the above operation on the edge $e$ until $e$ contains exactly $2$ vertices of degree greater than one, and on each other edges like $e$,
we finally arrive at a hypergraph $G^{(2)}$ such that the unique hypertree $T^{(2)}$ attached at $u_0$ is a power hypertree, and
$\la_{\min}(G^{(2)}) \le \la_{\min}(G^{(1)})$.

{\bf Case 3:} $T^{(2)}$ contains pendent edges not attached at $u_0$.
%Let $v$ be a vertex of $T^{(2)}$ such that $d(v)>2$ and $d(u_0,v)$ is largest among all vertices of $T^{(2)}$,
%where $d(u_0,v)$ denotes the distance of from $u_0$ to $v$.
%Then there are at least two hanging hyperpath attached at $v$.
Let $x$ be a first eigenvector of $G^{(2)}$.
We assert that $|x_{u_0}|=\max_{v \in V(G_0)}|x_v|$.
Otherwise, there exists a vertex $v_0$ of $G_0$ such that $|x_{v_0}|>|x_{u_0}|$.
Relocating $T^{(2)}$ rooted at $u_0$ and attaching to $v_0$, we will get a hypergraph $\tilde{G} \in \mathcal{T}_m(G_0)$ such
that $\la_{\min}(\tilde{G}) < \la_{\min}(G^{(2)})$ by Lemma \ref{movedge}.
Then $\la_{\min}(\tilde{G}) < \la_{\min}(G)$, a contradiction to $G$ being minimizing.

We also assert that $|x_{u_0}|=\max_{v \in V(G^{(2)})}|x_v|$.
Otherwise, let $w_0 \in V(T^{(2)}) \backslash \{u_0\}$ such that $|x_{w_0}|=\max_{v \in V(G^{(2)})}|x_v| > |x_{u_0}|$.
Then $w_0$ must have degree greater than one;
otherwise letting $e$ be the only edge containing $w_0$, and letting $\bar{w}_0$ a vertex in $e$ with degree greater than one,
relocating the branch attached at $\bar{w}_0$ and attaching to $w_0$, we will get a hypergraph with smaller least eigenvalue by Lemma \ref{movedge}, a contradiction.
Now relocating $G_0$ rooted at $u_0$ and attaching to $w_0$, we also  get a hypergraph with smaller least eigenvalue by Lemma \ref{movedge}.
By a similar discussion, $u_0$ is the unique vertex satisfying $|x_{u_0}|=\max_{v \in V(G^{(2)})}|x_v|$.

Finally let $e$ be a pendent edge of $T^{(2)}$ attached at $v \ne u_0$.
Relocating $e$ from $v$ to $u_0$, we get a hypergraph with smaller least eigenvalue by Lemma \ref{movedge}, a contradiction.
The result now follows.
\end{proof}

If taking $G_0$ to be a single edge in Theorem \ref{min}, then we get the following corollary immediately.
%the minimizing hypertree with $m$ edges is exactly the

\begin{coro}\label{hyptree-star}
Among all $k$-uniform hypertree with $m$ edges, the minimizing hypergraph is unique, namely the hyperstar $S_m^k$.
\end{coro}

It is easy to verify that a hypertree $T$ is odd-bipartite (by induction on the number of edges).
So the spectrum of $\A(T)$ is symmetric with respect to the origin \cite{SSW}, and hence $\la_{\min}(T)=-\rho(T)$.
Therefore, Corollary \ref{hyptree-star} implies a result of \cite{LSQ} when $k$ is even, which says that $S_m^k$ is the unique hypergraph with maximum spectral radius among all $k$-uniform hypertrees with $m$ edges.

Next we consider the case that $G_0$ is non-odd-bipartite.
Let $C_{2n+1}$ be an odd cycle of length $2n+1$ (as a simple graph).
Let $k \ge 4$ be a positive even integer.
Then $C_{2n+1}^{k, {k \over 2}}$ is a non-odd-bipartite $k$-uniform hypergraph \cite{KF}.
Let $K_n^k$ be a complete $k$-uniform hypergraph on $n \ge k+1$ vertices.
It is easy to verify that $K_n^k$ is non-odd-bipartite.
If taking $G_0$ be $C_{2n+1}^{k, {k \over 2}}$ or $K_n^k$ in Theorem \ref{min}, we get the following results immediately.

\begin{coro}
The minimizing hypergraph in $\mathcal{T}_m(C_{2n+1}^{k, {k \over 2}})$ is the unique hypergraph $C_{2n+1}^{k, {k \over 2}}(u) \diamond S_m^k(u)$ for
an arbitrary vertex $u$ of $C_{2n+1}^{k, {k \over 2}}$, up to isomorphism.
\end{coro}

\begin{coro}
The minimizing hypergraph in $\mathcal{T}_m(K_n^k)$ is the unique hypergraph $K_n^k(u) \diamond S_m^k(u)$ for
an arbitrary vertex $u$ of $K_n^k$, up to isomorphism.
\end{coro}

\bibliographystyle{amsplain}

\end{document}